\documentclass{article}
\usepackage{secondpreamble}
\usepackage{authblk}
\usepackage{todonotes}
\usepackage{xcolor}

\usepackage[
  backend=biber,
  style=numeric,
  sorting=none,
  giveninits=true
]{biblatex}
\renewbibmacro{in:}{}

\newcommand{\F}{\mathbb F}

\title{Cubes in the Torus}
\author{Douglas Barnes, Sean Jaffe\thanks{Department of Pure Mathematics and Mathematical Statistics, University of Cambridge. Emails: \texttt{db875@cam.ac.uk}, \texttt{scj47@cam.ac.uk}}}
\date{10th August 2026}

\begin{document}
\maketitle
\vspace{-2em}
\begin{abstract}
\noindent 

For $q> p$, let $T(n,q,p)$ be the minimum number of translates of the  cube \(\{0,1,\dots,p-1\}^n\) required to cover the $n$-dimensional torus $(\mathbb{Z}/q\mathbb{Z})^n$. We show that for each $q$ there exists a constant $1\le \Lambda_q \le 2$ such that $T(n,q,2)=(\Lambda_q + o(1))(q/2)^n$.
\end{abstract}
\section{Introduction}
For $q > p$, let $T(n,q,p)$ be the minimum number of translates of a $p\times p\times\dots \times p$ cube required to cover the $n$-dimensional torus
$(\mathbb{Z}/q\mathbb{Z})^n = \{0,1,\dots,q-1\}^n $.  Of course, we need at least $(q/p)^n$ cubes (because of a volume argument), and a simple probabilistic argument in which each possible translate of a  cube is selected independently with a suitable probability gives the best known upper bound of $T(n,q,p) = O(n(q/p)^n)$\cite{Bollobs2010,Bogdanov}. In this paper, we close this gap and calculate the true asymptotics whenever $p=2$:
\begin{theorem}\label{thm:main} For each $q\geq 3$, 
there is a  constant $1\le \Lambda_q\le 2$ such that
\[
T(n,q,2)=(\Lambda_q+o(1))\left(\frac{q}{2}\right)^n.
\]
Furthermore, when $q=3$, we have $\frac{7}{4}\leq \Lambda_3\leq 2$.
\end{theorem}
Here, the non-trivial lower bound of $T(n,3,2)\geq \left(\frac74 - o(1)\right)\left(\frac32\right)^n$ already follows from the work of Kolev~\cite{Kolev2013} and \"{O}sterg{\aa}rd and Riihonen~\cite{ annals}, thus our contribution is to prove the corresponding upper bound. When $(q,p)=(3,2)$, this problem is known as a football pool problem\cite{Kolev2013, annals,Footy}. In a football pool, each ticket records a prediction for each of \(n\) matches, with three possible outcomes for each match: home win, draw, or away win. A first prize is awarded to any ticket that correctly predicts all \(n\) outcomes, while a second prize is awarded to any ticket with exactly one incorrect prediction. Clearly, \(3^n\) tickets are necessary and sufficient to guarantee a first prize, while sphere-covering bounds show that \((1+o(1))3^n/(2n+1)\) tickets are necessary and sufficient to guarantee a second prize. Indeed, H{\"a}m{\"a}l{\"a}inen, Honkala,
             Litsyn and {\"O}sterg{\aa}rd, proved that one may cover the torus with at most \((1+o(1))3^n/(2n+1)\) sets of the form \(S_v = \{w\in \{0,1,2\}^n : w_i = v_i \text{ for all but at most one value of $i$} \}\) for $v\in \{0,1,2\}^n$ \cite{Game}. This paper gives exact asymptotics for the \emph{inverse football pool problem}, in which \(T(n,3,2)\) denotes the minimum number of tickets required to guarantee a ticket that makes an incorrect prediction for every one of the \(n\) matches.

 These discrete results also imply optimal bounds for coverings of the continuous torus. Let $\mu(n,\epsilon)$ denote the number of axis parallel $n$-dimensional cubes of side length $\epsilon$ required to cover $(\R/\Z)^n$, then $\mu(n,2/q)= T(n,q,2)$ \cite[Lemma 2]{Bogdanov}, hence Theorem \ref{thm:main} establishes the optimal asymptotics of $\mu(n,2/3)$. The generalisation of covering $(\Z_q)^n$ by $(\Z_p)^n$-cubes (or equivalently, $\mu(n,p/q)$) has been studied \cite{annals}, and in light of Theorem \ref{thm:main} we expect the ratio $T(n,q,p)/(q/p)^n$ to be bounded.  
 
\section{Statement on the use of AI}
Generative AI tools were used for all the proofs in this paper, which also provided a Lean formalisation \cite{lean}. All arguments were subsequently checked, edited, and presented by the authors, who take full responsibility for the correctness of the results.  No AI was used to write the introduction or commentary. 

\section{Proof of Theorem \ref{thm:main}}
 We split the proof of Theorem \ref{thm:main} into the following two parts, from which Theorem \ref{thm:main} follows. The first result (Theorem \ref{thm:main:first}) proves the result when $q = 3$, and Theorem \ref{thm:main:second} lifts this result to all $q\geq 3$. 
\begin{theorem}\label{thm:main:first} For all $n\geq 10$,
\[
\frac{7}{4}\left(\frac{3}{2}\right)^n\leq T(n,3,2)\leq 2\left(\frac{3}{2}\right)^n.
\]

\end{theorem}
The lower bound in Theorem \ref{thm:main:first} already follows from the literature and we do not prove it here. Indeed, a computer-assisted search showed that $T(9) = 68$ \cite{Kolev2013,computer}, which combined with the inequality $T(n)\geq \lceil \tfrac{3}{2}T(n-1)\rceil$ \cite[Theorem 3.3]{annals} yields the lower bound.
\begin{theorem}\label{thm:main:second}
Fix $p\ge 2$. Suppose
\(
q=ap+b(p+1)\) for some
\(
a,b\in\mathbb Z_{\ge 0}\) such that $q >p$. Then, for every $n\ge 0$,
\[
\left(\frac{q}{p}\right)^n\leq  T(n,q,p)\le
\left(\frac{q}{p+1}\right)^n
T(n,p+1,p).
\]
\end{theorem}
The lower bound in Theorem \ref{thm:main:second} follows immediately from a volume argument, so we only prove the upper bound. Before proceeding with the proofs of Theorems \ref{thm:main:first} and \ref{thm:main:second}, we deduce Theorem \ref{thm:main} from them:
\begin{proof}[Proof of Theorem \ref{thm:main}] Fix $q\geq3$, and define $\rho_n = T(n,q,2)/(q/2)^n$. Note that every $q\geq3$ can be expressed as $2a +3b$ for some $a,b\in \mathbb{Z}_{\geq 0}$ with $a$ and $b$ not both $0$. Thus, by Theorems \ref{thm:main:first} and \ref{thm:main:second}, $\rho_n$ is bounded above by 2. Furthermore, by \cite[Theorem 3.3]{annals}, $\rho_n$ is increasing, hence converges to a constant $\Lambda_{q}\leq 2 $ as $n\to\infty$. By the lower bounds in Theorems \ref{thm:main:first} and \ref{thm:main:second}, we have: $\Lambda_3\geq \frac{7}{4}$ and $\Lambda_q\geq 1$ for all $q$. This proves Theorem \ref{thm:main}.
\end{proof}


\section{Proof of Theorem \ref{thm:main:first}}
As discussed, we need only prove the upper bound in Theorem \ref{thm:main:first}.  For $s=(s_1,\dots,s_n)\in\{0,1,2\}^n$, define the $s$-translated cube
\[
Q_s=\{x\in\{0,1,2\}^n:x_i\neq s_i\text{ for every }i\}=s+\{1,2\}^n,
\]
For $x,y\in\{0,1,2\}^n$, we say $x$ covers $y$ if $y\in Q_x$. If $\mathcal{S}\subseteq \{0,1,2\}^n$, we say that $\mathcal{S}$ covers the torus if for all $x\in \{0,1,2\}^n$, there exists some $s\in \mathcal{S}$ such that $s$ covers $x$. Thus, it suffices to construct some $\mathcal{S}$ which covers the torus that satisfies $|\mathcal{S}|\leq 2 \left(\frac{3}{2}\right)^n$.

In the proof of Theorem \ref{thm:main:first}, we first identify each element of $(\mathbb{Z}/ 3 \mathbb{Z})^n$ with a vector in $(\mathbb{F}_{2}^2\backslash \{0\})^n$. We also define a bilinear form $[ \cdot,\cdot]$ on $(\mathbb{F}_2^2)^n$ with the following nice property. If $x  $ and $y$ are elements of $ (\mathbb{Z}/3\mathbb{Z})^n$ which are associated to vectors $v_x$ and $v_y$ respectively, then $y$ covers $x$ if and only if $[ v_x,v_y] = \one$. This trick reformulates the covering condition into a purely linear-algebraic statement. This allows us to prove that if $W$ is a vector subspace of $\mathbb{F}_2^{2n}$ satisfying certain properties, then $W\cap (\mathbb{F}_2^2\backslash \{0\})^n$ corresponds under the above identification to a cover of the torus. 

Next, we construct such subspaces $W$ from the adjacency matrices of graphs where every vertex has odd degree. Finally, by selecting a uniformly random graph where every degree is odd, we show that one of the resulting covers has size $\leq 2 \left(\frac32\right)^n$.

For convenience, let $T(n)=T(n,3,2)$.  We identify the three symbols $0,1,2$ of $\mathbb{Z}/3\mathbb{Z}$ with the three nonzero vectors of $\F_2^2$ as follows:
\[
0\longleftrightarrow(1,0),
\qquad
1\longleftrightarrow(0,1),
\qquad
2\longleftrightarrow(1,1).
\]

On $\F_2^2$, define the symmetric bilinear form $[u,v]=u_1v_2+u_2v_1$, which for nonzero $u,v\in\F_2^2$ satisfies
\[
[u,v]=
\begin{cases}
0,&u=v,\\
1,&u\neq v.
\end{cases}
\]

and note $x\in Q_s$ if and only if
\begin{equation}\label{eq:coordinate-cover}
[s_i,x_i]=1
\qquad\text{for every }i.
\end{equation}
Additionally, for $u = (u_1,\dots,u_n), v = (v_1,\dots,v_n)\in(\mathbb F_2^2)^n$, we extend this notation
coordinate-wise by setting
\[
[u,v]:=\bigl([u_1,v_1],\ldots,[u_n,v_n]\bigr)\in\mathbb F_2^n.
\]
Thus $x\in Q_s$ if and only if $[s,x] = \one$. 
We also define $\langle u,v\rangle= \one \cdot [u,v] = \sum_{i=1}^n [u_i,v_i]$, where $\cdot$ denotes the standard Euclidean inner product on $\F^n_2$. The following lemma shows that subspaces of $(\F_2^2)^n$ can give coverings of the torus. 

\begin{lemma}[Covering lemma]\label{lem:symplectic}
Let $W\le (\mathbb{F}_2^{2})^n$ be a vector subspace. Suppose $W$ satisfies $W = W^{\perp}$ and for every $w = (w_1,\dots,w_n)\in W$, the number of indices $i\in [n]$ for which $w_i \neq (0,0)$ is even. Then
\(
\mathcal S=W\cap(\F_2^2\setminus\{0\})^n
\)
is a cover of the torus $\{0,1,2\}^n$ under the encoding above.
\end{lemma}
\begin{proof}[Proof of Lemma \ref{lem:symplectic}]
Fix $x=(x_1,\dots,x_n)\in (\mathbb{F}_2^2\backslash\{0\})^n$. Consider the linear map \(L_x:W\to\F_2^n\) defined by
\[
L_x(w)=[w,x].
\]
We claim that $\one\in\image L_x$. Suppose not, then since $\image L_x$ is then a proper subspace of $\mathbb{F}_2^n$ not containing $\one$,
there is a vector $u\in\F_2^n$ such that
\begin{equation}\label{eq:separator}
u\cdot L_x(w)=0\quad\text{for every }w\in W,
\qquad
u\cdot\one=1.
\end{equation}
Define $y\in (\F_2^2)^n$ by
\(
y_i=u_i x_i.
\)
The first condition in \eqref{eq:separator} gives that $\langle w,y\rangle=0$ for every $w\in W$, so $y\in W^\perp=W$. For $i\in [n]$, since $x_i\neq (0,0)$, $y_i \neq (0,0)$ if and only if $u_i = 1$. Thus, by the second condition in \eqref{eq:separator}, $|\{i \in [n] : y_i \neq (0,0)\}|$ is odd. This contradicts the second hypothesis in the statement of the lemma. Hence some $w\in W$ satisfies $L_x(w)=\one$. Then
$[w,x]=\one$, so $w_i$ is nonzero and different from $x_i$ for each $i$.
Thus $w\in\mathcal S$, and $x\in Q_w$.
\end{proof}

We now describe how to construct a subspace $W_A$ of $(\mathbb{F}_2^2)^n$ satisfying the hypotheses of Lemma \ref{lem:symplectic} from a matrix $A$ over $\mathbb{F}_2$ with certain properties. 

Let $A$ be a $n\times n$ matrix over $\mathbb{F}_2$. Define 
\[
W_A=\{(a,Aa):a\in\F_2^n\}\subseteq(\F_2^2)^n,
\]
where $(a,Aa) = ((a_1,(Aa)_1),(a_2,(Aa)_2),\dots,(a_n,(Aa)_n))$ is an element of $(\mathbb{F}_2^2)^n$ because each $(a_i,(Aa)_i)\in \mathbb{F}_2^2$. Since $(a,Aa) + (b,Ab) = (a+b,A(a+b))$ for $a,b\in \mathbb{F}_2^n$, it is clear that $W_A$ is a vector subspace of $(\mathbb{F}_2^2)^n$. 


\begin{lemma}\label{proposition_self_dual_odd_degree_graphs}
    Suppose that $n\ge2$ is even and let $A$ be a symmetric $n\times n$ matrix over $\mathbb{F}_2$ with $A_{ii}=0$ for every \(i\in [n]\). If $A\one=\one$, then $W_A = W_A^{\perp} $ and for every $w = (w_1,\dots,w_n)\in W_A$, the number of indices $i\in [n]$ such that $w_i\neq (0,0)$ is even.
\end{lemma} 
We note that whenever $n\geq 2$ is even, such a matrix $A$ must always exist, as, for example, one may take the adjacency matrix of a perfect matching on $n$ vertices. 
\begin{proof}[Proof of Lemma \ref{proposition_self_dual_odd_degree_graphs}]

Since $A$ is symmetric,
\[
\langle(a,Aa),(b,Ab)\rangle
=a^TAb+(Aa)^Tb
=a^T(A+A^T)b=0.
\]
Thus $W_A \subseteq W_A^{\perp}$. It has dimension $n$, half the dimension of $(\F_2^2)^n$, so $W_A=W_A^\perp$.

It remains to verify that $|\{i \in [n] : w_i \neq (0,0)\}|= |\{i \in [n] : (a_i,(Aa)_i) \neq (0,0)\}|$ is even for every $w\in W$, where $a$ is the element of $\mathbb{F}_2^n$ that satisfies $w = (a,Aa)$. For a binary pair $(u,v)$,
\[
\mathbf 1_{\{(u,v)\neq(0,0)\}}
=u+v+uv\pmod2.
\]
Therefore the parity of $|\{i \in [n] : w_i \neq (0,0)\}|$
\[
\begin{aligned}
\sum_i\bigl(a_i+(Aa)_i+a_i(Aa)_i\bigr)
&=\one^Ta+\one^TAa+a^TAa\\
&=\one^Ta+(A\one)^Ta+a^TAa\\
&=\one^Ta+\one^Ta+0\\
&=0.
\end{aligned}
\]
Here $a^TAa=0$ because $A$ is symmetric with zero diagonal: every off-diagonal
term occurs twice and the diagonal terms vanish. 

\end{proof}
Let $n\geq 2$ be even, and suppose $G$ is a graph with vertex set $[n]$ such that every vetex of $G$ has odd degree. It is clear that the adjacency matrix $A$ of $G$ satisfies all the hypotheses in Lemma \ref{proposition_self_dual_odd_degree_graphs}. Thus, by Lemma~\ref{lem:symplectic} and Lemma \ref{proposition_self_dual_odd_degree_graphs}, 
\begin{equation}\label{eq:SA}
\mathcal S_A
=\bigl\{(a,Aa):\ (a_i,(Aa)_i)\neq(0,0)\text{ for every }i, \ a\in\F_2^n\bigr\} 
\end{equation}
covers the torus. In the next lemma we use random graph arguments to show there exists a graph $G$ for which $|\mathcal{S}_A|$ is small.

\begin{lemma}\label{lemmathreepointtwo} Suppose $n\geq 2$ is even. Then there exists a graph $G$ on $[n]$ such that every vertex has odd degree and the adjacency matrix $A$ satisfies $|\mathcal{S}_A|< 2\left(\frac{3}{2}\right)^n$.   
\end{lemma}

\begin{proof}[Proof of Lemma \ref{lemmathreepointtwo}]
Choose $G$ uniformly from the finite set of graphs on $[n]$ in which
every vertex has odd degree. Fix $a\in\F_2^n$ and write
\(
R=\{i:a_i=0\}\)

For $i\notin R$, $(a_i,(Aa)_i)\neq (0,0)$, so only the $i\in R$ impose constraints in the definition of $\mathcal{S}_A$  \eqref{eq:SA}. For $i\in R$,
it is nonzero exactly when
\begin{equation}\label{eq:cross-odd}
(Aa)_i=|N_G(i)\cap R^c|\equiv1\pmod2.
\end{equation}
One expects these $k$ parity conditions to roughly be independent fair coin tosses, but since $G$ is constrained to have odd degrees we must handle the dependence.  Suppose $\emptyset \neq R \subsetneq [n]$ and for $i\in R$ let $z_i = |N_G(i)\cap R^c| \mod 2$. Since every vertex in $R$ has odd degree,
\begin{equation}
    \sum_{i\in R} z_i= e(R,R^c) \equiv \sum_{i\in R} \deg_G (i) \equiv |R| \mod 2
\end{equation}
The vector $z=(z_i)_{i\in R}$ is uniformly distributed in the subset of $(\F_2)^{|R|}$ with coordinate sum equal to $|R| \mod 2$. Indeed, fix  $r \notin R$. If $i,j \in R$  are distinct, toggling the three edges of the triangle $\{r,i,j\}$ preserves the parity of every vertex while replacing $(z_i,z_j)\to (z_i+1,z_j+1)$. Any two vectors in $(\F_2)^{|R|}$ with the same coordinate sum parity differ in an even number of coordinates, so they can be bijectively mapped to each other by pairing those coordinates and applying triangle toggles. Hence 
\begin{equation}
    \mathbb{P}( (a,Aa) \in \mathcal{S}_A) = \mathbb{P} ( z_i = 1\text{ for every } i\in R)= 2^{-(|R|-1)}
\end{equation}
Thus,
\begin{equation}
    \mathbb{E}|\mathcal{S}_A| = 1+ \sum_{k=1}^{n-1} {n\choose k} 2^{-(k-1)} = 2(3/2)^n-1-2^{1-n}
\end{equation}
Therefore there exists some $G$ with every vertex of odd degree and adjacency matrix $A$ satisfying $|\mathcal{S}_A| <2(3/2)^n$. 

\end{proof}
We now conclude the proof of Theorem \ref{thm:main:first}. It follows from the previous lemma that for even $n$, $T(n) <2\left(\frac{3}{2}\right)^n $. This is also true for odd $n$ by projecting down from $n+1$  \cite[Theorem 3.3]{annals}, thereby completing the proof as desired.

\qed

\section{Proof of Theorem \ref{thm:main:second}}
As remarked upon previously, it suffices to only prove the upper bound.

Let $x,y\in\{0,1,\dots,q-1\}^n$. Similarly to before, say that $x$ covers $y$ for $(q,p)$ if $y\in x + \{1,\dots,p\}^n\subseteq (\mathbb{Z}/q\mathbb{Z})^n$. If $\mathcal{S}\subseteq (\mathbb{Z}/q\mathbb{Z})^n $, we say that $\mathcal{S}$ is a $(q,p)$-cover if for all $x\in \{0,1,\dots,q-1\}^n$, there exists a $y\in \mathcal{S}$ such that $y$ covers $x$ for $(q,p)$.

\begin{proof}[Proof of Theorem \ref{thm:main:second}]
Form a cyclic word of length $q$ by concatenating $a$ copies of
$0,1,\ldots,p-1$ and $b$ copies of $0,1,\ldots,p$; let
\[
\phi:\mathbb Z/q\mathbb Z
\to
\mathbb Z/(p+1)\mathbb Z
\]
be the resulting colouring. It is clear that every $p$ consecutive positions receive
distinct colours. Let $\mathcal{S}\subseteq(\mathbb Z/(p+1)\mathbb Z)^n$  be a minimal $(p+1,p)$-cover. Independently in each coordinate $i=1,\dots,n$ choose a
permutation $\pi_i$ of the $p+1$ colours uniformly at random and put
$\phi_i=\pi_i\circ\phi$. Each $\phi_i$ defines a colouring of $\mathbb{Z}/q\mathbb{Z}$ with $p+1$ colours. Define
\[
\widetilde{\mathcal{S}}
=
\bigl\{
t\in(\mathbb Z/q\mathbb Z)^n:
(\phi_1(t_1),\ldots,\phi_n(t_n))\in \mathcal{S}
\bigr\}.
\]

We claim for any choice of permutations $(\pi_i)_{i=1}^n$, $\widetilde{\mathcal{S}}$ is a $(q,p)$-cover. Fix
$x\in(\mathbb Z/q\mathbb Z)^n$. For each $i$, there are exactly $p$ intervals of length $p$ that contain $x_i$. These intervals are precisely $\{a,a+1,\dots,a+{p-1}\}$ for $a\in T_i$ where $T_i = \{x_i-p,x_i-p+1,\dots, x_i-1\}$ is the set of possible starts for such an interval. 

Each element of $T_i$ has a distinct colour according to $\phi_i$. Thus $\phi_i(T_i)$ forms a $p$-element subset of
$\mathbb Z/(p+1)\mathbb Z$, so by the pigeonhole principle for every $i$ there exists a unique  $y_i\in \mathbb{Z}/(p+1)\mathbb{Z}$ with $y_i \notin \phi_i(T_i)$. Since $\mathcal S$ is a $(p+1,p)$-cover, there exists some $s\in \mathcal{S}$ such that $s$ covers $y=(y_1,\ldots,y_n)$. In particular $s_i\neq y_i$ for each $i$, so $s_i\in \phi_i(T_i)$ i.e. there exists  $t_i\in T_i$ such that $\phi_i(t_i) = s_i$. It is clear that $(t_1,\dots,t_n)\in \widetilde{\mathcal{S}}$ and that $(t_1,\dots,t_n)$ covers $x$ for $(q,p)$. Thus $\widetilde{\mathcal{S}}$ is a $(q,p)$-cover.

Lastly, we argue by probability that there exist permutations $(\pi_i)_{i=1}^n$ with $|\widetilde{\mathcal{S}}|$ small. For fixed $t = (t_1,\dots,t_n)\in (\mathbb{Z}/q\mathbb{Z})^n$, the random vector
\(
(\phi_1(t_1),\ldots,\phi_n(t_n))
\)
is distributed uniformly on $(\mathbb Z/(p+1)\mathbb Z)^n$. Hence $\mathbb{P}(t\in \widetilde{\mathcal{S}}) = \frac{|\mathcal{S}|}{(p+1)^n}$ so 
\[
\mathbb E|\widetilde{\mathcal{S}}|
=
q^n\frac{|\mathcal S|}{(p+1)^n}.
\]
Thus there exists a choice of permutations $(\pi_i)_{i=1}^n$ which makes $|\widetilde{\mathcal{S}}|\leq \big(\frac{q}{p+1}\big)^n |\mathcal{S}|=\left(\frac{q}{p+1}\right)^n T(n,p+1,p) $. Thus,
\[T(n,q,p)\leq 
\big(\frac{q}{p+1}\big)^n T(n,p+1,p) 
\]

\end{proof}

\section{Acknowledgements}
The authors would like to thank Imre Leader for introducing the problem to them. The authors thank Julian Sahasrabudhe for helpful suggestions concerning the presentation of this manuscript.

\printbibliography
\end{document}